\documentclass{amsart}
\usepackage[round,authoryear]{natbib}
\usepackage{tabu}
\usepackage{hyperref}

\usepackage{amssymb} 
\usepackage{enumitem}
\usepackage{multirow}
\setlist[itemize]{leftmargin=12mm}
\setlist[enumerate]{leftmargin=12mm}
\usepackage{comment}
\usepackage{url}

\DeclareMathOperator{\Rad}{Rad}
\DeclareMathOperator{\Det}{det}
\DeclareMathOperator{\Aut}{Aut}
\DeclareMathOperator{\GL}{GL}
\DeclareMathOperator{\SL}{SL}
\DeclareMathOperator{\Ker}{Ker}
\DeclareMathOperator{\Gal}{Gal}
\DeclareMathOperator{\norm}{Norm}
\DeclareMathOperator{\ord}{\upsilon}

\newcommand{\Q}{\mathbb{Q}}
\newcommand{\Z}{\mathbb{Z}}

\newcommand{\F}{\mathbb{F}}
\newcommand{\cA}{\mathcal{A}}
\newcommand{\cD}{\mathcal{D}}
\newcommand{\cE}{\mathcal{E}}
\newcommand{\cK}{\mathcal{K}}
\newcommand{\cN}{\mathcal{N}}
\newcommand{\OO}{\mathcal{O}}
\newcommand{\ga}{\mathfrak{a}}
\newcommand{\gb}{{\mathfrak{b}}}

\newcommand{\ff}{\mathfrak{f}}

\newcommand{\fp}{\mathfrak{p}}
\newcommand{\fq}{\mathfrak{q}}

\newcommand{\mP}{\mathfrak{P}}     

\newcommand{\sB}{\mathfrak{B}}

\begin{document}

\newtheorem{thm}{Theorem}
\newtheorem{lem}{Lemma}[section]
\newtheorem{prop}[lem]{Proposition}
\newtheorem{cor}[lem]{Corollary}

\title[Modular elliptic curves]{
Modular elliptic curves over real abelian fields\\
and the generalized Fermat equation
$x^{2\ell}+y^{2m}=z^p$
}

\author{Samuele Anni and Samir Siksek}

\address{Mathematics Institute\\
	University of Warwick\\
Coventry\\
	CV4 7AL \\
	United Kingdom}
\email{samuele.anni@gmail.com}
\email{samir.siksek@gmail.com}

\date{\today}
\thanks{The authors are supported by EPSRC Programme Grant 
\lq LMF: L-Functions and Modular Forms\rq\  EP/K034383/1.
}


\begin{abstract}
Let $K$ be a real abelian field of odd class number in which $5$
is unramified. Let $S_5$ be the set of places of $K$ above $5$.
Suppose for every non-empty proper subset $S \subset S_5$ there
is a totally positive unit $u \in \OO_K$ such that
$\prod_{\fq \in S} \norm_{\F_\fq/\F_5}(u \bmod{\fq}) \ne \overline{1}$.
We prove that every semistable elliptic curve over $K$ is modular,
using a combination of several powerful modularity theorems
and class field theory. We deduce that if $K$ is a real abelian
field of conductor $n<100$, with $5 \nmid n$ and $n \ne 29$, $87$, $89$,
then every semistable elliptic curve $E$ over $K$ is modular.

Let $\ell$, $m$, $p$ be prime, with $\ell$, $m \ge 5$ and $p \ge 3$.
To a putative non-trivial primitive
solution of the generalized Fermat $x^{2\ell}+y^{2m}=z^p$
we associate a Frey elliptic curve defined over $\Q(\zeta_p)^+$,
and study its mod $\ell$ representation with the help
of level lowering and our modularity result. 
We deduce the non-existence of non-trivial primitive solutions
if $p \le 11$, or if $p=13$ and $\ell$, $m \ne 7$.
\end{abstract}
\maketitle

\section{Introduction}

Let $p$, $q$, $r \in \Z_{\geq 2}$. The equation
\begin{equation}\label{eqn:FCgen}
x^p+y^q=z^r
\end{equation}
is known as the \textbf{generalized Fermat equation} 
(or the \textbf{Fermat--Catalan equation})
 with signature $(p,q,r)$.
As in Fermat's Last Theorem, one is interested in integer solutions
$x$, $y$, $z$. Such a solution is called \textbf{non-trivial} if
$xyz \neq 0$, and \textbf{primitive} if $x$, $y$, $z$ are coprime.
Let $\chi=p^{-1}+q^{-1}+r^{-1}$. 
The \textbf{generalized Fermat conjecture} 
\citep{DG,Da97},
also known as the Tijdeman--Zagier conjecture
and as the Beal conjecture \citep{Beukers},
is concerned with the case $\chi<1$.
It states that the only non-trivial primitive solutions to
\eqref{eqn:FCgen} with $\chi<1$ are
\begin{gather*}
1+2^3 = 3^2, \quad 2^5+7^2 = 3^4, \quad 7^3+13^2 = 2^9, \quad
2^7+17^3 = 71^2, \\
3^5+11^4 = 122^2, \quad 17^7+76271^3 = 21063928^2, \quad
1414^3+2213459^2 = 65^7, \\
9262^3+15312283^2 = 113^7, \; \,
43^8+96222^3 = 30042907^2, \; \, 33^8+1549034^2 = 15613^3.
\end{gather*}
The generalized Fermat conjecture
has been established for many signatures $(p,q,r)$,
including for several infinite families of signatures,
starting with
Fermat's Last Theorem $(p,p,p)$ by
\citet{Wiles};
$(p,p,2)$ and $(p,p,3)$ by \citet{DM};
$(2,4,p)$ by \citet{El} and 
\citet*{BEN};
$(2p,2p,5)$ by \citet{Bennett}; $(2,6,p)$
by \cite{BC}; and other signatures by other researchers.
An excellent, exhaustive and up-to-date survey was recently compiled by
\citet*{BennettSurvey}, which also proves
the generalized Fermat conjecture for several families of signatures,
including $(2p,4,3)$.

The main Diophantine result of this paper is the following theorem.
\begin{thm}\label{thm:1}
Let $p=3$, $5$, $7$, $11$ or $13$. Let $\ell$, $m \ge 5$ be primes,
and if $p=13$ suppose moreover that $\ell$, $m \ne 7$.
Then the only primitive solutions to
\begin{equation}\label{eqn:main}
x^{2\ell}+y^{2m}=z^p, 
\end{equation}
are the trivial ones 
$(x,y,z)=(\pm 1, 0, 1)$ and $(0, \pm 1, 1)$.
\end{thm}
If $\ell$, $m$ is $2$ or $3$ then \eqref{eqn:main} has
no non-trivial primitive solutions for prime $p \ge 3$;
this follows from the aforementioned work  
on Fermat equations of signatures $(2,4,p)$, $(2,6,p)$ and $(2p,4,3)$.

\medskip

Our approach is unusual in that it treats several bi-infinite
families of signatures. 
We start with a descent argument (Section~\ref{sec:Descent}), inspired
by the approach of \citet{Bennett} for $x^{2n}+y^{2n}=z^5$ and that of \citet{recipes}
for $x^r+y^r=z^p$ with certain small values of $r$. For $p=3$ the descent argument allows us to 
quickly obtain a contradiction (Section~\ref{sec:peq3}) 
through results of \citet{BS}.
The bulk of the paper is devoted to $5 \le p \le 13$.
Our descent allows us to construct Frey curves 
(Sections~\ref{sec:FreyCurve},~\ref{sec:FreyCurve2}) 
attached to \eqref{eqn:main} that
are defined over the real cyclotomic field $K=\Q(\zeta+\zeta^{-1})$
where $\zeta$ is a $p$-th root of unity, 
or, for $p \equiv 1 \pmod{4}$, defined
over the unique subfield $K^\prime$ of $K$ of degree $(p-1)/4$. These Frey curves
are semistable over $K$, though not necessarily over $K^\prime$.

In the remainder of the paper
we study the 
mod $\ell$ representations of these Frey curves using modularity
and level lowering. 
Several recent papers \citep{DF,FS,FSsmall,recipes,BDMS}
apply modularity and level lowering over 
totally real fields to study Diophantine problems.
We need to refine many of the ideas in those papers, 
both because we are dealing with representations over number fields
of relatively high degree, and because we 
are aiming for a \lq clean\rq\ result without any exceptions (the methods
are much easier to apply for sufficiently large $\ell$).
We first establish modularity of the Frey curves by combining
a modularity theorem for residually reducible representations due to \citet{SkinnerWiles}
with a theorem of \citet{Thorne} for residually dihedral representations,
and implicitly applying modularity lifting theorems of \citet{Kisin} and others for representations
with \lq big image\rq. We shall use class field theory to glue together these great modularity theorems
and produce our own theorem (proved in Section~\ref{sec:modularity})
that applies to our Frey curves, but which we expect to be 
of independent interest.
\begin{thm} \label{thm:modularity}
Let $K$ be a real abelian number field. 
Write $S_5$ for the prime ideals $\fq$ of $K$ above $5$.
Suppose
\begin{enumerate}
\item[(a)] $5$ is unramified in $K$;
\item[(b)] the class number of $K$ is odd;
\item[(c)] for each non-empty proper subset $S$ of $S_5$,
 there is some totally positive unit $u$ of $\OO_K$
such that 
\begin{equation}\label{eqn:conditionc}
\prod_{\fq \in S} \norm_{\F_\fq/\F_5}(u \bmod{\fq}) \ne \overline{1} \, .
\end{equation}
\end{enumerate}
Then every semistable elliptic curve $E$ over $K$
is modular.
\end{thm}
Theorem~\ref{thm:modularity} allows us to deduce the following
corollary (also proved in Section~\ref{sec:modularity}).
\begin{cor}\label{cor:modularity}
Let $K$ be a real abelian field of conductor $n < 100$ 
with $5 \nmid n$ and $n \ne 29$, $87$, $89$. Let $E$ be a
semistable elliptic curve over $K$. Then $E$ is modular.
\end{cor}

To apply level lowering theorems to a modular mod $\ell$ representation, one
must first show that this representation is irreducible.
Let $G_K=\Gal(\overline{K}/K)$.
The mod $\ell$ representation that concerns us, denoted $\overline{\rho}_{E,\ell} \; : \; G_K \rightarrow \GL_2(\F_\ell)$,
is the one attached to the $\ell$-torsion of 
our semistable Frey elliptic curve $E$ defined over the field 
$K=\Q(\zeta+\zeta^{-1})$ of degree $(p-1)/2$.
We shall exploit semistability of our Frey curve
 to show, with the help of class field theory,
 that if $\overline{\rho}_{E,\ell}$ is reducible 
then $E$ or some $\ell$-isogenous curve possesses non-trivial $K$-rational 
$\ell$-torsion. Using famous results on torsion of elliptic curves
over number fields of 
small degree due to \citet{Kamienny,Parent1,Parent2,DKSS}
and some computations of $K$-points on certain modular curves,
we prove the required irreducibility result 
(Section~\ref{sec:Irred}). 

The final step (Section~\ref{sec:final}) 
in the proof of Theorem~\ref{thm:1} requires computations 
of certain Hilbert eigenforms
over the fields $K$ together with their eigenvalues at primes of small norm. 
For these computations we have made use of the 
\lq Hilbert modular forms package\rq\ developed by
Demb\'el\'e, Donnelly, Greenberg and Voight and available within the
\texttt{Magma} computer algebra system \citep{Magma}. For the theory behind
this package see  
\citep{DV}. 
For $p \ge 17$,
the required computations
are beyond the capabilities of current software,
though the strategy for proving Theorem~\ref{thm:1}
should be applicable to larger $p$ once these
computational limitations are overcome. 
In fact, at the end of Section~\ref{sec:final}, we heuristically argue  
that
the larger the value
of $p$ is, the more likely that the argument used to complete
the proof of Theorem~\ref{thm:1} will succeed for that particular $p$.
We content ourselves with proving (Section~\ref{sec:proof2})
the following theorem.
\begin{thm}\label{thm:2}
Let $p$ be an odd prime, and
let $K=\Q(\zeta+\zeta^{-1})$ where $\zeta=\exp(2 \pi i/p)$.
Write $\OO_K$ for the ring of integers in $K$ and
$\fp$ for the unique prime ideal above $p$.
Suppose that there are no
elliptic curves $E/K$ with full $2$-torsion and
conductors $2 \OO_K$, $2 \fp$. Then there is an
ineffective constant $C_p$ (depending only on $p$)
such that for all primes 
$\ell$, $m \ge C_p$, the only primitive solutions to
\eqref{eqn:main}
are the trivial ones 
$(x,y,z)=(\pm 1, 0, 1)$ and $(0, \pm 1, 1)$.

If $p \equiv 1 \pmod{4}$ then let $K^\prime$ be the unique
subfield of $K$ of degree $(p-1)/4$. Let $\sB$ be
the unique prime ideal of $K^\prime$ above $p$.
Suppose that there are no
elliptic curves $E/K^\prime$ with non-trivial $2$-torsion and
conductors $2 \sB$, $2 \sB^2$. Then there is an
ineffective constant $C_p$ (depending only on $p$)
such that for all primes 
$\ell$, $m \ge C_p$, the only primitive solutions to
\eqref{eqn:main}
are the trivial ones 
$(x,y,z)=(\pm 1, 0, 1)$ and $(0, \pm 1, 1)$.
\end{thm}

\medskip

The computations described in this paper were carried out using the 
computer algebra \texttt{Magma} \citep{Magma}. The code and output
is available from: \newline
\url{http://homepages.warwick.ac.uk/~maseap/progs/diophantine/}

\medskip

We are grateful to the three referees for careful reading of
the paper and for suggesting many improvements. We are
indebted to 
Lassina Demb\'{e}l\'{e},
Steve Donnelly, 
Marc Masdeu and Jack Thorne
for stimulating conversations.

\section{Proof of Theorem~\ref{thm:modularity} and Corollary~\ref{cor:modularity}}
\label{sec:modularity}
We shall need a result from class field theory. The following
version is proved by \citet[Appendice A]{KrausNF}.
\begin{prop}\label{prop:Krauscf}
Let $K$ be a number field, and $q$ a rational
prime that does not ramify in $K$. 
Denote the mod $q$ cyclotomic character by
$\chi_q \; : \; G_K \rightarrow \F_q^\times$.
Write $S_q$ for 
the set of primes $\fq$ of $K$ above $q$, and let $S$
be a subset of $S_q$. Let 
$\varphi \; : \; G_K \rightarrow \F_q^\times$ be a character satisfying: 
\begin{itemize}
\item[(a)] $\varphi$ is unramified outside $S$ and the infinite places of $K$;
\item[(b)] $\varphi \vert_{I_\fq} = \chi_q \vert_{I_\fq}$
for all $\fq \in S$; here $I_\fq$ denotes the inertia
subgroup of $G_K$ at $\fq$. 
\end{itemize}
Let $u \in \OO_K$
be a unit that is positive in each real embedding of $K$.
Then 
\[
\prod_{\fq \in S} \norm_{\F_{\fq}/\F_q} (u \bmod \fq)=\overline{1}.
\]
\end{prop}
\begin{proof}
For the reader's convenience we give a sketch of Kraus's elegant
argument.
Let $L$ be the cyclic field extension of $K$ cut out by the kernel of $\varphi$.
Then we may view $\varphi$ as a character $\Gal(L/K) \rightarrow \F_q^\times$.
Write $M_K$ for the places of $K$. 
For  $\upsilon \in M_K$,
let $\Theta_\upsilon : K_\upsilon^* \rightarrow \Gal(L/K)$
be the local Artin map. Let $u \in \OO_K$ be a unit that is positive in each
real embedding. 
We shall consider the values $\varphi(\Theta_\upsilon(u)) \in \F_q^\times$ as $\upsilon$
ranges over $M_K$.

Suppose first that $\upsilon \in M_K$ is infinite. If $\upsilon$ is complex
then $\Theta_\upsilon$ is trivial and so certainly 
$\varphi(\Theta_\upsilon(u))=\overline{1}$ in $\F_q^\times$. So suppose $\upsilon$ is real.
As $u$ is positive in $K_\upsilon$, it is a local norm
and hence in the kernel of $\Theta_\upsilon$. 
Therefore $\varphi(\Theta_\upsilon(u))=\overline{1}$.

Suppose next that $\upsilon \in M_K$ is finite.  
As $u \in \OO_\upsilon^\times$, it follows from local reciprocity
that $\Theta_\upsilon(u)$ belongs to the inertia subgroup $I_\upsilon \subseteq \Gal(L/K)$.
If $\upsilon \notin S$ then $\varphi(I_\upsilon)=1$ by (a)
and so $\varphi(\Theta_\upsilon(u))=\overline{1}$. 
Thus suppose that $\upsilon=\fq \in S$. It follows from (b) that
$\varphi(\Theta_\fq(u))=\chi_q(\Theta_\fq(u))$. Through an explicit
calculation, \citet[Appendice A, Proposition 1]{KrausNF} shows that
$\chi_q(\Theta_\fq(u))= \norm_{\F_\fq/\F_q}(u \bmod{\fq})^{-1}$.

Finally, by global reciprocity, $\prod_{\upsilon \in M_K} \Theta_\upsilon(u)=1$.
Applying $\varphi$ to this equality completes the proof.
\end{proof}

We shall also make use of the following theorem of 
\citet[Theorem 1.1]{Thorne}. 
\begin{thm}[Thorne]
Let $E$ be an elliptic curve
over a totally real field $K$. Suppose $5$
is not a square in $K$ and that $E$ has no
$5$-isogenies defined over $K$. Then $E$ is modular.
\end{thm}
Thorne deduces this result by combining his beautiful modularity 
theorem for residually dihedral representations \citep[Theorem 1.2]{Thorne},
with \citep*[Theorem 3]{FLS}. The latter result is essentially
a straightforward consequence of the powerful modularity lifting theorems
for residual representations with \lq big image\rq\
due to \citet{Kisin}, \citet*{BGG1,BGG2}, and \citet{BD}.

Finally we shall need the following modularity theorem 
for residually reducible representations due to
\citet[Theorem A]{SkinnerWiles}.
\begin{thm}[Skinner and Wiles]
Let $K$ be a real abelian number field. Let $q$ be
an odd prime, and
\[
\rho : G_K \rightarrow \GL_2(\overline{\Q}_q)
\]
be a continuous, irreducible representation,
unramified away from a finite number of places of $K$.
Suppose $\overline{\rho}$ is reducible and write
$\overline{\rho}^{\mathrm{ss}} = \psi_1 \oplus \psi_2$. Suppose further that
 \begin{enumerate}
   \item[(i)] the splitting field $K(\psi_1 / \psi_2)$ of $\psi_1 / \psi_2$
is abelian over $\Q$;
   \item[(ii)] $(\psi_1 / \psi_2)(\tau) = -1$ for each complex 
conjugation $\tau$;
    \item[(iii)]
$(\psi_1 / \psi_2)\vert_{D_\fq}\neq 1$ for each $\fq \mid q$;
   \item[(iv)] for all $\fq \mid q$,
\[
\rho \vert_{D_\fq} \sim 
\begin{pmatrix}
\phi_1^{(\fq)} \cdot \tilde{\psi}_1 & * \\
0 & \phi_2^{(\fq)} \cdot  \tilde{\psi}_2\\
\end{pmatrix}
\]
with $\phi_2^{(\fq)}$ factoring through a pro-$q$ extension 
of $K_\fq$ and $\phi_2^{(\fq)} \vert_{I_\fq}$
having finite order, and 
where $\tilde{\psi}_i$ is a Teichm\"{u}ller lift of $\psi_i$;
   \item[(v)] $\Det(\rho) = \psi \chi_q^{k-1}$, where $\psi$ is
a character of finite order, and $k \ge 2$ is an integer.
 \end{enumerate}
Then the representation $\rho$ is associated to a Hilbert modular newform.
\end{thm}

\subsection*{Proof of Theorem~\ref{thm:modularity}}
As $5$ is unramified in $K$, it certainly is not a square in $K$.
If $E$ has no $5$-isogenies defined over $K$ then the result follows
from Thorne's theorem. We may thus suppose that the mod $5$
representation $\overline{\rho}$ of $E$ is reducible,
and write $\overline{\rho}^{\mathrm{ss}}=\psi_1 \oplus \psi_2$.
We will verify hypotheses (i)--(v) in the theorem of Skinner and Wiles 
(with $q=5$)
to deduce the modularity
of $\rho : G_K \rightarrow \Aut(T_5(E)) \cong \GL_2(\Z_5)$,
where $T_5(E)$ is the $5$-adic Tate module of $E$.
If $E$ has good supersingular reduction at some
$\fq \mid 5$ then (as $\fq$ is unramified) $\overline{\rho} \vert_{I_\fq}$
is irreducible \citep[Proposition 12]{Serre72}, contradicting the reducibility of $\overline{\rho}$. It follows that $E$ has good ordinary
or multiplicative reduction at all $\fq \mid 5$. In particular,
hypothesis (iv) holds with $\phi_i^{(\fq)}=1$.

Now $\psi_1 \psi_2 = \Det(\rho)=\chi_5$ so hypothesis (v) holds
with $\psi=1$ and $k=2$.
Moreover, for each complex conjugation $\tau$,
we have $(\psi_1/\psi_2)(\tau)=\psi_1(\tau) \psi_2(\tau^{-1})
=\psi_1(\tau)\psi_2(\tau)=\chi_5(\tau)=-1$ so (ii)
is satisfied.
It follows from the fact that $E$ has good ordinary 
or multiplicative at all $\fq \mid 5$, that 
$(\overline{\rho} \vert_{I_\fq})^{\mathrm{ss}} =  
\chi_5 \vert_{I_\fq} \oplus 1$ 
and so $\psi_1/\psi_2$ is non-trivial when restricted to $I_\fq$
(again as $\fq$ is unramified in $K$); this proves (iii).

It remains to verify (i). Note that $\psi_1/\psi_2=\chi_5/ \psi_2^2$.
Hence $K(\psi_1/\psi_2)$ is contained in the compositum
of the fields $K(\zeta_5)$ and $K(\psi_2^2)$,
and by symmetry also contained in the compositum of
the fields $K(\zeta_5)$ and $K(\psi_1^2)$.
It is sufficient to show that either 
$K(\psi_2^2)=K$ or $K(\psi_1^2)=K$.
Note that $\psi_i^2 \; : \; G_K \rightarrow \F_5^\times$ are quadratic characters
that are unramified at all 
archimedean places. We will show that one of them
is everywhere unramified, and then the desired result
follows from the assumption that the class number of $K$ 
is odd.
First note, by the semistability of $E$, that $\psi_1$
and $\psi_2$ are unramified at all finite primes $\fp \nmid 5$.
Let $S$ be the subset of $\fq \in S_5$ such that $\psi_1$
is unramified at $\fq$. By the above, we know that
these are precisely the $\fq \in S_5$ such that
$\psi_2 \vert_{I_\fq}=\chi_5 \vert_{I_\fq}$.
By assumption (c) and Proposition~\ref{prop:Krauscf},
we have that either $S=\emptyset$ or $S=S_5$.
If $S=\emptyset$ then $\psi_2$ is unramified
at all $\fq \mid 5$, and if $S=S_5$
then $\psi_1$ is unramified at all $\fq \mid 5$.
This completes the proof.

\subsection*{Proof of Corollary~\ref{cor:modularity}}
Suppose first that $K=\Q(\zeta_n)^+$. 
If $n \equiv 2 \pmod{4}$ then $\Q(\zeta_n)=\Q(\zeta_{n/2})$,
so we adopt the
usual convention of supposing that $n \not \equiv 2 \pmod{4}$.
We consider values $n<100$ and impose the restriction $5 \nmid n$,
which ensures that condition (a) of Theorem~\ref{thm:modularity}
is satisfied.  It is known \citep{Miller} that the
class number $h_n^+$ of $K$ is $1$ for all $n < 100$. Thus condition (b)
is also satisfied. Write $E_n^+$ for the group of 
units of $K$ and $C_n^+$ for the subgroup of cyclotomic
units. A result of \citet{Sinnott} asserts that 
$[E_n^+:C_n^+]=b \cdot h_n^+$ where $b$ is an explicit constant 
that happens to be $1$ for $n$ with at most $3$ distinct prime
divisors, and so certainly for all $n$ is our range. It follows that 
$E_n^+=C_n^+$ for $n < 100$. Now let $S_5$ be as in the statement
of Theorem~\ref{thm:modularity}. We wrote a simple \texttt{Magma}
script which for each $n <100$ satisfying $5 \nmid n$
and $n \not \equiv 2 \pmod{4}$ writes down a basis for 
the cyclotomic units $C_n^+$ and deduces a basis for 
the totally positive units. It then checks, for every
non-empty proper subset of $S_5$, if there is an element
$u$ of this basis of totally positive units that satisfies
\eqref{eqn:conditionc}. We found this to be the case for
all $n$ under consideration except $n=29$, $87$ and $89$.
The corollary follows from Theorem~\ref{thm:modularity}
for $K=\Q(\zeta_n)^+$ with 
$n$ as in the statement of the corollary.

Now let $K$ be a real abelian field with conductor $n$
as in the statement of the corollary. Then 
$K \subseteq \Q(\zeta_n)^+$. 
As $\Q(\zeta_n)^+/K$
is cyclic, modularity of an elliptic curve $E/K$
follows, by Langlands' cyclic base change theorem
\citep{Langlands}, from
modularity of $E$ over $\Q(\zeta_n)^+$,  completing the proof
of the corollary.

\section{Cyclotomic Preliminaries}
Throughout $p$ will be an odd prime. 
Let $\zeta$ be a primitive $p$-th root of unity, 
and
$K=\Q(\zeta+\zeta^{-1})$ the maximal real subfield of $\Q(\zeta)$.
We write
\[
\theta_j=\zeta^j+\zeta^{-j} \in K, \qquad j=1,\dotsc,(p-1)/2 \, .
\]
Let $\OO_K$ be the ring of integers of $K$. Let $\fp$ 
be the unique prime ideal of $K$ above $p$. Then $p \OO_K=\fp^{(p-1)/2}$.
\begin{lem}\label{lem:cyc}
For $j=1,\dotsc,(p-1)/2$, we have
\[
\theta_j \in \OO_K^\times, \qquad \theta_j+2 \in \OO_K^\times,
\qquad
(\theta_j-2) \OO_K=\fp.
\]
Moreover, 
$(\theta_j-\theta_k) \OO_K=\fp$
for
$1 \le j< k \le (p-1)/2$.
\end{lem}
\begin{proof}
Observe that 
$\theta_j=(\zeta^{2j}-\zeta^{-2j})/(\zeta^j-\zeta^{-j})$
and thus belongs to the group of cyclotomic units.
Given $j$,
let $j \equiv 2 r \pmod{p}$. Then $\theta_j+2=\theta_r^2 \in \OO_K^\times$.

For now, let $L=\Q(\zeta)$. 
Let $\mP$ be the prime of $\OO_L$ above $\fp$.
Then $\fp \OO_L=\mP^2$. As is well-known, $\mP=(1-\zeta^u) \OO_L$
for $u=1,2,\dotsc,p-1$. Note that
$\theta_j-2=(\zeta^r-\zeta^{-r})^2$, with $j \equiv 2 r \pmod{p}$, from which we deduce that
$(\theta_j-2) \OO_L=\mP^2=\fp \OO_L$, hence $(\theta_j-2) \OO_K=\fp$.

For the final part, $j \not \equiv \pm k \pmod{p}$. Thus there exist 
$u$, $v \not \equiv 0 \pmod{p}$
such that
\[
u+v \equiv j, \qquad u-v \equiv k \pmod{p}.
\]
Then
\[
(\zeta^u-\zeta^{-u})(\zeta^v-\zeta^{-v})=\theta_j-\theta_k
\]
and so $(\theta_j-\theta_k)\OO_L=\mP^2=\fp \OO_L$. This completes the proof.
\end{proof}

\section{The Descent}\label{sec:Descent}
Now let $\ell$, $m \ge 5$ be prime, and let $(x,y,z)$
be a non-trivial, primitive solution to \eqref{eqn:main}.
If $\ell=p$, then \eqref{eqn:main} can be rewritten as
$z^p+(-x^2)^p=(y^m)^2$. \citet{DM} have shown
that the only primitive solutions to the generalized Fermat equation \eqref{eqn:FCgen}
with signature $(p,p,2)$ are the trivial ones, giving us a contradiction.
We shall henceforth suppose that $\ell \ne p$ and $m \ne p$.

Clearly $z$ is odd. 
By swapping in \eqref{eqn:main} 
the terms $x^\ell$ and $y^m$ if necessary,
we may suppose that $2 \mid x$.
We factor the left-hand side over $\Z[i]$. It follows from
our assumptions that the two
factors $(x^\ell+y^m i )$ and $(x^\ell-y^m i)$ are coprime.
There exist coprime rational integers $a$, $b$ such that
\[
x^\ell+ y^m i=(a+bi)^p \, , \qquad z=a^2+b^2.
\]
Then
\begin{align*}
x^\ell &= \frac{1}{2} \left( (a+bi)^p+(a-bi)^p \right)\\
&= a \cdot \prod_{j=1}^{p-1} 
\left( (a+bi) +(a-bi) \zeta^j \right)\\
&= a \cdot \prod_{j=1}^{(p-1)/2}
\left(
(a+bi)+(a-bi) \zeta^j
\right) \cdot
\left(
(a+bi)+(a-bi) \zeta^{-j}
\right) \, .
\end{align*}
In the last step we have paired up the complex conjugate factors.
Multiplying out these pairs we obtain a factorization
of $x^\ell$ over $\OO_K$:
\begin{equation} \label{eqn:factor}
x^\ell= a \cdot \prod_{j=1}^{(p-1)/2}
\left(
(\theta_j+2) a^2+ (\theta_j-2) b^2
\right) \, .
\end{equation}
To ease notation, write
\begin{equation}\label{eqn:beta}
\beta_j= (\theta_j+2) a^2 + (\theta_j-2)b^2 \, , \qquad j=1,\dotsc,\frac{p-1}{2} \, .
\end{equation}

\begin{lem}\label{lem:factor}
Write $n=\ord_2(x) \ge 1$.  
\begin{enumerate}
\item[(i)] If $p \nmid x$ then
\[
a= 2^{\ell n} \alpha^\ell,
\qquad 
\beta_j \OO_K=\gb_j^\ell
\]
where $\alpha$ is a rational integer, and
$\alpha \OO_K,\gb_1,\dotsc,\gb_{(p-1)/2}$ are pairwise coprime ideals of $\OO_K$, all of which are coprime to $2p$.
\item[(ii)] If
$p \mid x$ then
\[
a = 2^{\ell n} p^{\kappa \ell  -1} \alpha^\ell,
\qquad \beta_j \OO_K=\fp \cdot \gb_j^\ell
\]
where $\kappa=\ord_p(x) \ge 1$, $\alpha$ is a rational integer and 
$\alpha \cdot \OO_K,\gb_1,\dotsc,\gb_{(p-1)/2}$ are pairwise coprime ideals of $\OO_K$, all of which are coprime to $2p$.
\end{enumerate}
\end{lem}
\begin{proof}
As $z=a^2+b^2$ is odd, exactly one of $a$, $b$ is even.
Thus the $\beta_j$ are coprime to $2 \OO_K$. We see from \eqref{eqn:factor}
that $2^{\ell n} \mid\mid a$, and hence that $b$ is odd.

As $a$, $b$ are coprime,
it is clear that the greatest common divisor of $a \OO_K$
and $\beta_j \OO_K$ divides $(\theta_j-2) \OO_K=\fp$.
Moreover, for $k \ne j$, the greatest common
divisor of $\beta_j \OO_K$ and 
$\beta_k \OO_K$ divides
\[
\left((\theta_j+2)(\theta_k-2)-(\theta_k+2)(\theta_j-2) \right) \OO_K
=4 (\theta_k-\theta_j) \OO_K= 4 \fp.
\]
However, $\beta_j$ is odd, and so the greatest common
divisor of $\beta_j \OO_K$ and
$\beta_k \OO_K$ divides $\fp$.
Now part (i) 
of the lemma follows immediately from \eqref{eqn:factor}.
So suppose $p \mid x$. For part (ii) we have to
check that $\fp \mid\mid \beta_j$. However, since 
$(\theta_j-2)\OO_K=\fp$, and $\theta_j+2 \in \OO_K^\times$,
reducing \eqref{eqn:factor} modulo $\fp$ shows that $a^p \equiv 0 \pmod{\fp}$,
and hence that $p \mid a$. Since $a$, $b$ are coprime, it follows
that $\ord_\fp(\beta_j)=1$. Now, from \eqref{eqn:factor}, 
\[
\frac{(p-1)}{2} \ord_p(a)
=\ord_\fp(a)=\ell \ord_\fp(x) - \sum_{j=1}^{(p-1)/2} \ord_\fp(\beta_j) 
=\frac{(p-1)}{2} (\kappa \ell -1)
\]
giving the desired exponent of $p$ in the factorization of $a$.
\end{proof}

\section{Proof of Theorem~\ref{thm:1} for $p=3$}\label{sec:peq3}
Suppose $p=3$. Then $K=\Q$
 and $\theta:=\theta_1=-1$.
We treat first the case $3 \nmid x$. 
By Lemma~\ref{lem:factor},
\[
a=2^{\ell n}\alpha^\ell, \qquad a^2-3b^2=\gamma^\ell
\]
for some coprime odd rational integers $\alpha$ and $\gamma$.
We obtain the equation 
\[
2^{2\ell n}\alpha^{2\ell}-\gamma^\ell=3b^2.
\]
\citet[Theorem 1]{BS} show that equation $x^n+y^n=3z^2$
has no solutions in coprime integers $x$, $y$, $z$ for $n \ge 4$,
giving us a contradiction.

We now treat $3 \mid x$. By Lemma~\ref{lem:factor},
\[
a=2^{\ell n} 3^{\kappa \ell -1} \alpha^\ell, \qquad a^2-3b^2=3\gamma^\ell
\]
for coprime rational integers $\alpha$, $\gamma$ that are also coprime to $6$.
Thus
\[
2^{2\ell n} 3^{2\kappa \ell -3} \alpha^{2\ell}-\gamma^\ell=b^2.
\]
Using the recipes of \citet[Sections 2, 3]{BS} we can attach a Frey 
curve to such a triple $(\alpha,\gamma,b)$ whose mod $\ell$
representation arises from a classical newform of weight $2$
and level $6$. As there are no such newforms our contradiction
is complete. 

\section{The Frey Curve}\label{sec:FreyCurve}

We shall henceforth suppose $p \ge 5$. 
From now on, fix $1 \le j$, $k \le (p-1)/2$ with $j \ne k$.
The expressions $\beta_j$, $\beta_k$ shall be given by \eqref{eqn:beta}.
For each such choice of $(j,k)$ we shall construct a Frey curve.
The idea is that the three expressions $a^2$, $\beta_j$, $\beta_k$
are roughly $\ell$-th powers (Lemma~\ref{lem:factor}). Moreover they
are linear combinations of $a^2$ and $b^2$, and hence must be
linearly dependent. 
Writing down this linear relation gives a Fermat equation (with coefficients) of signature $(\ell,\ell,\ell)$.
As in the work Hellegouarch, Frey, Serre, Ribet, Kraus, and many others, one can associate to
such an equation a Frey elliptic curve whose mod $\ell$ representation has very little ramification.
In what follows we take care to scale the expressions $a^2$, $\beta_j$, $\beta_k$
appropriately so that the Frey curve is semistable.

\medskip

\noindent \textbf{Case I: $p \nmid x$}. 
Let
\begin{equation}
\label{eqn:uvw}
u= 
\beta_j ,  \qquad
v=- \frac{(\theta_j-2)}{(\theta_k-2)} \beta_k , \qquad
w=\frac{4 (\theta_j-\theta_k)}{(\theta_k-2)} \cdot a^2.
\end{equation}
Then
$u+v+w=0$.
Moreover, by Lemmas~\ref{lem:cyc} and~\ref{lem:factor},
\[
u \OO_K=\gb_j^\ell, \qquad
v \OO_K=\gb_k^\ell, \qquad
w \OO_K= 2^{2 \ell n+2} \cdot \alpha^{2 \ell} \OO_K.
\]
We will let the Frey curve be
\begin{equation}\label{eqn:Frey}
E=E_{j,k} \; : \; Y^2=X(X-u)(X+v).
\end{equation}
For a non-zero ideal $\ga$, we define its \textbf{radical}, denoted by $\Rad(\ga)$,
to be the product of the distinct prime ideal factors of $\ga$.
\begin{lem}\label{lem:Frey1}
Suppose $p \nmid x$. Let $E$ be the Frey curve \eqref{eqn:Frey}
where $u$, $v$, $w$ are given by \eqref{eqn:uvw}.
The curve $E$ is semistable, with
multiplicative reduction at all primes above $2$
and good reduction at $\fp$. It has 
minimal discriminant 
and conductor
\[
\cD_{E/K}
=
2^{4 \ell n-4} {\alpha}^{4 \ell} \gb_j^{2\ell} \gb_k^{2\ell}, \qquad
\cN_{E/K}= 2 \cdot \Rad(\alpha \gb_j \gb_k).
\]
\end{lem}
\begin{proof}
The invariants $c_4$, $c_6$, $\Delta$, $j(E)$ have their usual meanings
and are given by:
\begin{equation}\label{eqn:inv}
\begin{gathered}
c_4=16(u^2-vw)=16(v^2-wu)=16(w^2-uv),\\
c_6=-32(u-v)(v-w)(w-u), \qquad 
\Delta=16 u^2 v^2 w^2, \qquad j(E)=c_4^3/\Delta \, .
\end{gathered}
\end{equation}
By Lemma~\ref{lem:factor}, we have $\alpha \OO_K$, $\gb_j$ and $\gb_k$
are pairwise coprime, and all coprime to $2p$.
In particular $\fp \nmid \Delta$ and so $E$ has good reduction
at $\fp$. Moreover, $c_4$ and $\Delta$ are coprime away from $2$.
Hence the model in \eqref{eqn:Frey} is already semistable away from $2$.
Recall that $2^\ell \mid a$ and $2 \nmid b$. Thus 
\[
u \equiv (\theta_j-2) b^2 \pmod{2^{2\ell}}, \quad
v \equiv -(\theta_j-2) b^2 \pmod{2^{2\ell}}, \quad
w \equiv 0 \pmod{2^{2\ell+2}}.
\]
It is clear that $\ord_\fq(j)<0$ for all $\fq \mid 2$. Thus $E$
has potentially multiplicative reduction at all $\fq \mid 2$.
Write $\gamma=-c_4/c_6$.
To show that $E$ has multiplicative reduction at $\fq$ 
it is enough to show that $K_\fq(\sqrt{\gamma})/K_\fq$
is an unramified extension \citep[Exercise V.5.11]{SilII}.
However, 
\[
c_4/16=(u^2-vw) \equiv (\theta_j-2)^2 \cdot b^4 \pmod {2^{2\ell}}
\]
which shows that $c_4$ is a square in $K_\fq$. Moreover,
\[
-c_6/16=2 (u-v)(v-w)(w-u)
\equiv 4 \cdot (\theta_j-2)^3 \cdot b^6 \pmod{2^{2\ell+1}} \, .
\]
Thus $K_\fq(\sqrt{\gamma})=K_\fq(\sqrt{\theta_j-2})$.
As before, letting $r$ satisfy $2r \equiv j \pmod{p}$,
we have $\theta_j-2=(\zeta^r-\zeta^{-r})^2$ and so
$K_\fq(\sqrt{\gamma})$ is contained in the unramified extension $K_\fq(\zeta)$.
Hence $E$ has multiplicative reduction at $\fq \mid 2$.

Finally $2$ is unramified in $K$, and so $\ord_\fq(c_4)=\ord_2(16)=4$.
Hence $\cD_{E/K}=(\Delta/2^{12}) \cdot \OO_K$ as required.
\end{proof}

\medskip

\noindent \textbf{Case II: $p \mid x$}. 
Let
\begin{equation}
\label{eqn:uvwp}
u= 
\frac{\beta_j}{(\theta_j-2)}, \qquad
v=- \frac{\beta_k}{(\theta_k-2)}, \qquad
w=\frac{4 (\theta_j-\theta_k)}{(\theta_j-2)(\theta_k-2)} \cdot a^2.
\end{equation}
Then, from Lemmas~\ref{lem:cyc} and~\ref{lem:factor},
\[
u \OO_K=\gb_j^\ell, \qquad
v \OO_K=\gb_k^\ell, \qquad
w \OO_K= 2^{2 \ell n+2} \cdot \fp^{\delta} \cdot \alpha^{2 \ell} \OO_K,
\]
where 
\begin{equation}\label{eqn:delta}
\delta=(\kappa \ell-1)(p-1)-1 \, .
\end{equation}
Again $u+v+w=0$ and the Frey curve is given by \eqref{eqn:Frey}.
\begin{lem}\label{lem:Frey2}
Suppose $p \mid x$. Let $E$ be the Frey curve \eqref{eqn:Frey}
where $u$, $v$, $w$ are given by \eqref{eqn:uvwp}. 
The curve $E$ is semistable, with
multiplicative reduction at $\fp$ and at 
all primes above $2$. It has 
minimal discriminant 
and conductor
\[
\cD_{E/K}
=
2^{4 \ell n-4} \fp^{2\delta} {\alpha}^{4 \ell} \gb_j^{2\ell} \gb_k^{2\ell},
\qquad \cN_{E/K}=2 \fp \cdot \Rad(\alpha \gb_j \gb_k).
\]
\end{lem}
\begin{proof}
The proof is an
easy modification of the proof of Lemma~\ref{lem:Frey1}.
\end{proof}

\section{A closer look at the Frey Curve for $p \equiv 1 \pmod{4}$}
\label{sec:FreyCurve2}

In this section we shall suppose that $p \equiv 1 \pmod{4}$. 
The Galois group of $K=\Q(\zeta+\zeta^{-1})$ is cyclic
of order $(p-1)/2$. Thus
the field $K=\Q(\zeta+\zeta^{-1})$ has a unique involution 
$\tau \in \Gal(K/\Q)$ and we let $K^\prime$ be the 
subfield of degree $(p-1)/4$ that is fixed by this involution.
In the previous section we let $1 \le j$, $k \le (p-1)/2$ 
with $j \ne k$. In this section we shall impose the further
condition that $\tau(\theta_j)=\theta_k$. 
Now a glance at the definition \eqref{eqn:Frey} of the Frey curve $E$
and the formulae \eqref{eqn:uvwp} for $u$ and $v$
in the case $p \mid x$ shows that the curve $E$ is in fact defined over 
$K^\prime$. This is not true in the case $p \nmid x$, but we can take
a twist of the Frey curve so that it is defined over $K^\prime$.

\bigskip

\noindent \textbf{Case I: $p \nmid x$}.

Let
\begin{equation}\label{eqn:uvwd}
u^\prime= 
(\theta_k-2)\beta_j ,  \qquad
v^\prime=- (\theta_j-2)\beta_k, \qquad 
w^\prime=4 (\theta_j-\theta_k) \cdot a^2,
\end{equation}
and let
\[
E^\prime \; : \; Y^2=X(X-u^\prime)(X+v^\prime) \, .
\] 
Clearly the coefficients of
$E^\prime$ are invariant under $\tau$ and so $E^\prime$
is defined over $K^\prime$. Moreover, $E^\prime/K$ 
is the quadratic twist of $E/K$ by $(\theta_k-2)$.
Let $\sB$ be the unique prime of $K^\prime$ above $p$.
Let 
\[
\gb_{j,k}=\norm_{K/K^\prime}(\gb_j)
=\norm_{K/K^\prime}(\gb_k) \, .
\]
It follows from Lemma~\ref{lem:factor} that the $\OO_{K^\prime}$-ideal
$\gb_{j,k}$ is coprime to $\alpha$ and to $2p$.
An easy calculation leads us to the following lemma.
\begin{lem}\label{lem:Frey3}
Suppose $p \nmid x$. Let $E^\prime/K^\prime$ be the above Frey elliptic curve.
Then $E^\prime$ is semistable away from $\sB$,
with minimal discriminant and conductor
\[
\cD_{E^\prime/K^\prime}
=
2^{4 \ell n-4} \sB^3 {\alpha}^{4 \ell} \gb_{j,k}^{2\ell}, \qquad
\cN_{E^\prime/K^\prime}= 2 \cdot \sB^2 \cdot \Rad(\alpha \gb_{j,k}).
\]  
\end{lem}

\bigskip

\noindent \textbf{Case II: $p \mid x$}.

Another straightforward computation yields the following lemma.
\begin{lem}\label{lem:Frey4}
Suppose $p \mid x$. Let $E^\prime=E$ be the Frey curve in Lemma~\ref{lem:Frey2}.
Then $E^\prime$ is defined $K^\prime$.
The curve $E^\prime/K^\prime$ is semistable
 with minimal discriminant and conductor 
\[
\cD_{E^\prime/K^\prime}
=
2^{4 \ell n-4} \sB^\delta {\alpha}^{4 \ell} \gb_{j,k}^{2\ell}, \qquad
\cN_{E^\prime/K^\prime}= 2 \cdot \sB \cdot \Rad(\alpha \gb_{j,k}),
\]
where $\delta$ is given by \eqref{eqn:delta}.
\end{lem}

\noindent \textbf{Remark.} Clearly $E$
has full $2$-torsion over $K$. 
The curve $E^\prime$
has a point of order $2$ over $K^\prime$, but not 
necessarily full $2$-torsion.

\section{Proof of Theorem~\ref{thm:2}}\label{sec:proof2}

\begin{lem}\label{lem:modularity}
Let $p$ be an odd prime. There is an ineffective
 constant $C_p^{(1)}$ depending on $p$ such
that for odd primes $\ell$, $m \ge C_p^{(1)}$,
and any non-trivial primitive solution $(x,y,z)$
of \eqref{eqn:main}, the Frey curve $E/K$ as  
in Section~\ref{sec:FreyCurve} is modular. If $p \equiv 1 \pmod{4}$
then under the same assumptions, the Frey curve $E^\prime/K^\prime$ as in Section~\ref{sec:FreyCurve2} is modular. 
\end{lem}
\begin{proof}
\citet*{FLS} show that for any totally real field $K$
there are at most finitely many non-modular $j$-invariants.
Let $j_1,\dotsc,j_r$ be the values of these $j$-invariants.
Let $\fq$ be a prime of $K$ above $2$. By Lemmas~\ref{lem:Frey1}
and~\ref{lem:Frey2},
we have $\ord_\fq(j(E)) =-(4 \ell n -4)$ with $n \ge 1$. Thus for $\ell$, $m$
 sufficiently large we have $\ord_\fq(j(E)) < \ord_\fq(j_i)$
for $i=1,\dotsc,r$, completing the proof.
\end{proof}
\noindent \textbf{Remarks.}
\begin{itemize}[leftmargin=5mm]
\item The argument in \citet*{FLS} relies on Faltings' Theorem
(finiteness of the number of rational points on a curve of genus $\ge 2$)
to deduce finiteness
of the list of possibly non-modular $j$-invariants.
It is for this reason that the constant $C_p^{(1)}$
(and hence the constant $C_p$ in Theorem~\ref{thm:2})
is ineffective.
\item In the above argument, it seems that it is enough
to suppose that $\ell$ is sufficiently large without
an assumption on $m$. However, in Section~\ref{sec:Descent}
we swapped the terms $x^{2\ell}$ and $y^{2m}$ in \eqref{eqn:main}
 if needed
to ensure that $x$ is even. Thus in the above argument 
we need to suppose that 
both $\ell$ and $m$ are sufficiently large.
\end{itemize}

We shall make use of the following  
result due to \citet[Theorem 2]{FSirred}.
It is a variant of results proved by
 \citet{KrausNF} and by \citet{David}. All these build
on the celebrated uniform boundedness theorem of \citet{Merel}.
\begin{thm}\label{thm:irred}
Let $K$ be a totally real field. There is an effectively
computable constant $C_K$ such that for a prime $\ell>C_K$,
and for an elliptic curve $E/K$ semistable at all $\lambda \mid \ell$,
the mod $\ell$ representation $\overline{\rho}_{E,\ell} \, : \, G_K \rightarrow
\GL_2(\F_\ell)$ is irreducible. 
\end{thm}
In \citet[Theorem 2]{FSirred} it is assumed that $K$ is
Galois as well as totally real. Theorem~\ref{thm:irred}
follows immediately on replacing $K$ with its Galois closure.

\begin{lem}\label{lem:ll}
Let $E/K$ be the Frey curve given in Section~\ref{sec:FreyCurve}.
Suppose $\overline{\rho}_{E,\ell}$
is irreducible and $E$ is modular. Then
$\overline{\rho}_{E,\ell}\sim \overline{\rho}_{\ff,\lambda}$
for some Hilbert cuspidal eigenform $\ff$ over $K$ of parallel weight $2$
that is new at level $\cN_\ell$, where
\[
\cN_\ell=
\begin{cases}
2 \OO_K & \text{if $p \nmid x$} \\
2 \fp & \text{if $p \mid x$} \, .
\end{cases}
\]
Here $\lambda \mid \ell$ is a prime of $\Q_\ff$, the field
generated over $\Q$ by the eigenvalues of $\ff$.

If $p \equiv 1 \pmod{4}$, let $E^\prime/K^\prime$ be the Frey
curve given in Section~\ref{sec:FreyCurve2}.
Suppose $\overline{\rho}_{E^\prime,\ell}$
is irreducible and $E$ is modular. Then
$\overline{\rho}_{E^\prime,\ell}\sim \overline{\rho}_{\ff,\lambda}$
for some Hilbert cuspidal eigenform $\ff$ over $K$ of parallel weight $2$
that is new at level $\cN^\prime_\ell$, where
\[
\cN^\prime_\ell=
\begin{cases}
2 \sB^2 & \text{if $p \nmid x$} \\
2 \sB & \text{if $p \mid x$} \, .
\end{cases}
\]
\end{lem}
\begin{proof}
This immediate from Lemmas~\ref{lem:Frey1}, \ref{lem:Frey2},
\ref{lem:Frey3} and \ref{lem:Frey4},
and a standard level lowering recipe derived in \cite[Section 2.3]{FS}
from the work of Jarvis, Fujiwara and Rajaei. Alternatively,
one could use modern modularity lifting theorems which
integrate level lowering with modularity lifting, as for example
in \cite{BD}.
\end{proof}

\subsection*{Proof of Theorem~\ref{thm:2}}
Let $K=\Q(\zeta+\zeta^{-1})$ and $E$ be the Frey curve constructed
in Section~\ref{sec:FreyCurve}. Let $C^{(1)}_p$ be the constant
in Lemma~\ref{lem:modularity}, and
$C^{(2)}_p=C_K$ be the constant in Theorem~\ref{thm:irred}. 
Let $C_p=\max(C^{(1)}_p, C^{(2)}_p)$.
Suppose that $\ell$, $m \ge C_p$.
Then $\overline{\rho}_{E,\ell}$ is irreducible and modular,
and it follows from Lemma~\ref{lem:ll} that 
$\overline{\rho}_{E,\ell}\sim \overline{\rho}_{\ff,\lambda}$
for some Hilbert eigenform over $K$ of parallel weight $2$
that is new at level $\cN_\ell$, where $\cN_\ell=2  \OO_K$
or $2  \fp$. 
Now a standard argument (c.f. \citet*[Section 4]{BS}
or \citet[Section 3]{Kra97} or \citet[Section 9]{IHP}) shows that,
after enlarging $C_p$ by an effective amount, we may
suppose that the field of eigenvalues of $\ff$ is $\Q$.
Observe that the level $\cN_\ell$ is non-square-full (meaning there
is a prime $\fq$ at which the level has valuation $1$).
As the level is non-square-full and the field of 
eigenvalue is $\Q$, the eigenform
$\ff$ is known to correspond to some
elliptic curve $E_1/K$ of conductor $\cN_\ell$ \citep{Blasius},
and $\overline{\rho}_{E,\ell} \sim \overline{\rho}_{E_1,\ell}$.
Finally, and again by standard arguments (loc.\ cit.),
we may enlarge $C_p$ by an effective constant so that the isomorphism 
$\overline{\rho}_{E,\ell} \sim \overline{\rho}_{E_1,\ell}$ forces 
$E_1$ to either have full $2$-torsion, or to be isogenous
to an elliptic curve $E_2/K$ that has full $2$-torsion.
This contradicts the hypothesis of Theorem~\ref{thm:2} that there
are no such elliptic curves with conductor $2\OO_K$, $2 \fp$,
and completes the proof of the first part of the theorem.
The proof of the second part is similar, and makes use
of the Frey curve $E^\prime/K^\prime$.

\section{Modularity of the Frey Curve For $5 \le p \le 13$}

\begin{lem}\label{lem:modularitysmallp}
If $p=5$, $7$, $11$ or $13$ then the Frey curve $E/K$
in Section~\ref{sec:FreyCurve} is modular.
If $p=5$, $13$ then the Frey curve $E^\prime/K^\prime$
in Section~\ref{sec:FreyCurve2}
is modular.
\end{lem}
\begin{proof}
Recall that $E$ is defined over $K=\Q(\zeta+\zeta^{-1})$
where $\zeta$ is a primitive $p$-th root of unity. 
If $p=5$ then $K=\Q(\sqrt{5})$, and modularity
of elliptic curves over real quadratic fields was recently 
established by \citet*{FLS}.

For $p=7$, $11$, $13$, the prime $5$ is unramified in $K$,
the class number of $K$
is $1$, and condition (c) of Theorem~\ref{thm:modularity} is easily
verified. Thus $E$ is modular.

For $p=13$, the curves $E$ and $E^\prime$ are at 
worst quadratic twists over $K$,
and $K/K^\prime$ is quadratic. The modularity of $E^\prime/K^\prime$
follows from the modularity of $E/K$
and the cyclic base change theorem of \citep{Langlands}.
For $p=5$ we could use the same argument, or more simply note
that $K^\prime=\Q$, and conclude by the modularity
theorem over the rationals \citep*{BCDT}.
\end{proof}

\section{Irreducibility of $\overline{\rho}_{E,\ell}$ for $5 \le p \le 13$}\label{sec:Irred}
We let $E$ be the Frey curve as in Section~\ref{sec:FreyCurve},
and $p=5$, $7$, $11$, $13$. To
apply Lemma~\ref{lem:ll} we need to prove the
irreducibility of $\overline{\rho}_{E,\ell}$ for $\ell \ge 5$;
equivalently, we need to show that $E$ does not have an $\ell$
isogeny for $\ell \ge 5$. Alas, there is not yet a uniform boundedness
theorem for isogenies. 
The papers of \cite{KrausNF},  \cite{David}, \cite{FSirred}
do give effective bounds $C_K$ such that for $\ell>C_K$ the representation
$\overline{\rho}_{E,\ell}$ is irreducible, however these bounds 
are too large for our present purpose. We will refine the arguments
in those papers making use of the fact that the curve $E$ is semistable,
and the number fields $K=\Q(\zeta+\zeta^{-1})$ all have narrow class number $1$.
Before doing this, we relate, for $p=5$ and $13$,
 the representations $\overline{\rho}_{E,\ell}$
and $\overline{\rho}_{E^\prime,\ell}$ where $E^\prime$ is the Frey
curve in Section~\ref{sec:FreyCurve2}.
\begin{lem}\label{lem:compare}
Suppose $p=5$ or $13$. Let $\tau$ be the unique
involution of $K$ and $K^\prime$ the subfield fixed by it.
Let $j$ and $k$ satisfy $\tau(\theta_j)=\theta_k$. Let
$E/K$ be the Frey elliptic curve in Section~\ref{sec:FreyCurve}
and $E^\prime/K^\prime$ the 
Frey curve in Section~\ref{sec:FreyCurve2}, associated to this
pair $(j,k)$. Then 
$\overline{\rho}_{E,\ell}$ is irreducible as a representation
of $G_K$ if and only if 
$\overline{\rho}_{E^\prime,\ell}$ is irreducible as a representation
of $G_{K^\prime}$.
\end{lem}
\begin{proof}
Note that $K/K^\prime$ is a quadratic extension and $E/K$ is a quadratic
twist of $E^\prime/K$.  Thus $\overline{\rho}_{E,\ell}$ is a twist
of $\overline{\rho}_{E^\prime,\ell} \vert_{G_K}$ by a quadratic character.
If $\overline{\rho}_{E^\prime,\ell}$ is reducible as a
representation of $G_{K^\prime}$ then certainly
$\overline{\rho}_{E,\ell}$ is reducible as a representation of $G_K$.

Conversely, suppose $\overline{\rho}_{E^\prime,\ell}(G_{K^\prime})$ is
irreducible. We would like to show that
$\overline{\rho}_{E,\ell}(G_K)$ is irreducible.
It is enough to show that $\overline{\rho}_{E^\prime,\ell}(G_K)$
is irreducible. Let $\fq \mid 2$ be a prime of $K^\prime$. 
Then
$\ord_\fq(j(E^\prime))=4-4\ell n$ which is negative but not divisible
by $\ell$. Thus $\overline{\rho}_{E^\prime,\ell}(G_{K^\prime})$
contains an element of order $\ell$ \citep[Proposition V.6.1]{SilII}.
By Dickson's classification \citep{SwD} 
of subgroups of $\GL_2(\F_\ell)$ we see
that $\overline{\rho}_{E^\prime,\ell}(G_{K^\prime})$
must contain $\SL_2(\F_\ell)$. The latter is a simple group,
and must therefore be contained in 
$\overline{\rho}_{E^\prime,\ell}(G_K)$. This completes the proof.
\end{proof}

\begin{lem}\label{lem:tors}
Suppose $\overline{\rho}_{E,\ell}$ is reducible. Then
either $E/K$ has non-trivial $\ell$-torsion, or is $\ell$-isogenous
to an elliptic curve defined over $K$ that has non-trivial
$\ell$-torsion.
\end{lem}
\begin{proof}
Suppose $\overline{\rho}_{E,\ell}$ is reducible, and write
\[
\overline{\rho}_{E,\ell}\sim 
\begin{pmatrix}
\psi_1 & * \\
0 & \psi_2 
\end{pmatrix} \, .
\]
We shall show that either $\psi_1$ or $\psi_2$ is trivial.
It follows in the former case that $E$ has non-trivial 
$\ell$-torsion, and in the latter case that the $K$-isogenous
curve $E/\Ker(\psi_1)$ has non-trivial $\ell$-torsion.

As $K$ has narrow class number $1$ for $p=5$, $7$, $11$, $13$,
it is sufficient to show that one of $\psi_1$, $\psi_2$ is
unramified at all finite places. 
As $E$ is semistable, the characters $\psi_1$ and $\psi_2$
are unramified away from $\ell$ and the infinite places.
Let $S_\ell$ be the set of primes $\lambda \mid \ell$ of $K$. 
Let $S \subset S_\ell$ for the set of $\lambda \in S_\ell$
such that $\psi_1$ is ramified at $\lambda$. Then (c.f.
proof of Theorem~\ref{thm:modularity}) $\psi_2$
is ramified exactly at the primes $S\setminus S_\ell$. Moreover,
$\psi_1 \vert_{I_\lambda}= \chi_\ell \vert_{I_\lambda}$
for all $\lambda \in S$, and
$\psi_2 \vert_{I_\lambda}= \chi_\ell \vert_{I_\lambda}$
for all $\lambda \in S_\ell \setminus S$.
It is enough to show that either $S$ is empty  or $S_\ell \setminus S$ is empty.

Suppose $S$ is a non-empty proper subset of $S_\ell$.
Fix $\lambda \in S$ and let $D=D_\lambda \subset G=\Gal(K/\Q)$ be the 
decomposition group of $\lambda$; by definition 
$\lambda^\sigma=\lambda$ for all $\sigma \in D_\lambda$. 
As $K/\Q$ is abelian and Galois, $D_{\lambda^\prime}=D$
for all $\lambda^\prime \in S_\ell$, and $G/D$ acts
transitively and freely on $S_\ell$.
Fix a set $T$ of coset representatives for $G/D$.
Then there is a subset $T^\prime \subset T$
such that 
\[
S=\{ \lambda^{\tau^{-1}}\; :\; \tau \in T^\prime\},
\qquad
S_\ell\setminus S=\{ \lambda^{\tau^{-1}} \; : \; \tau \in T \setminus T^\prime\}.
\]
As $S$ is a non-empty proper subset of $S_\ell$, we have that
$T^\prime$ is a non-empty proper subset of $T$.
Now, by Proposition~\ref{prop:Krauscf}, for any totally positive unit $u$
of $\OO_K$,
\[
\prod_{\tau \in T^\prime} \norm_{\F_\lambda/\F_\ell} (u+\lambda^{\tau^{-1}})
=\overline{1} \, .
\]
But
\begin{equation*}
\begin{split}
\norm_{\F_\lambda/\F_\ell} (u+\lambda^{\tau^{-1}})
&=\prod_{\sigma \in D} (u+\lambda^{\tau^{-1}})^\sigma\\
&=\prod_{\sigma \in D} (u^\sigma+\lambda^{\tau^{-1}})\\
&= \left( 
\prod_{\sigma \in D} (u^{\sigma \tau}+\lambda)
\right)^{\tau^{-1}} \\
&= 
\prod_{\sigma \in D} (u^{\sigma \tau}+\lambda)
\qquad \text{(as this expression belongs to $\F_\ell$).} 
\end{split}
\end{equation*}
Let
\[
B_{T^\prime,D}(u)=\norm_{K/\Q}\left( \left(\prod_{\tau \in T^\prime,~\sigma \in D} u^{\sigma \tau}
\right) -1 \right) \, .
\]
It follows that $\ell \mid B_{T^\prime,D}(u)$. Now let $u_1,\dotsc,u_d$ be a system
of totally positive units. Then $\ell$ divides 
\[
B_{T^\prime,D}(u_1,\dotsc,u_d)=\gcd \left( B_{T^\prime,D}(u_1),\dotsc,B_{T^\prime,D}(u_d)  \right).
\]
To sum up, if the lemma is false for $\ell$, then there is some subgroup $D$ of $G$
and some non-empty proper subset $T^\prime$ of $G/D$ such that $\ell$ divides $B_{T^\prime,D}(u_1,\dotsc,u_d)$.

The proof of the lemma is completed by a computation that we now describe.
For each of
$p=5$, $7$, $11$, $13$ we fix a basis 
$u_1,\dotsc,u_d$ for the system of totally positive units
of $\OO_K$. 
We run through the subgroups $D$ of $G=\Gal(K/\Q)$.
For each subgroup $D$ we fix a set of coset representatives $T$, and run
through the non-empty proper subsets $T^\prime$ of $T$, computing $B_{T^\prime,D}(u_1,\dotsc,u_d)$.
We found that for $p=5$, $7$ the possible values for $B_{T^\prime,D}(u_1,\dotsc,u_d)$ are all $1$; 
for $p=11$ they are $1$ and $23$; and for $p=13$ they are $1$, $5^2$ and $3^5$. Thus the proof
is complete for $p=5$, $7$ and it remains to deal with $(p,\ell)=(11,23)$, $(13,5)$.
For each of these possibilities we ran through the non-empty proper
$S \subset S_\ell$ and checked that there is some totally positive unit
$u$ such that $\prod_{\lambda \in S} \norm(u+\lambda) \ne \overline{1}$.
This completes the proof.
\end{proof}

Suppose $\overline{\rho}_{E,\ell}$ is reducible. It follows
from Lemma~\ref{lem:tors} that there is an elliptic curve
$E_1/K$ (which is either $E$ or $\ell$-isogenous to $E$)
such that $E_1(K)$ has a subgroup isomorphic to 
$\Z/2\Z\times \Z/{2\ell \Z}$. Such an elliptic curve
is isogenous~\footnote{At the suggestion
of one of the referees we prove this statement.
Let $P_1$, $P_2 \in E_1(K)$ be independent points of order $2$.
Let $Q$ be a solution to the equation $2X=P_1$. Then
$Q$ has order $4$ and the complete set of solutions is
$\{ Q, Q+P_2, 3Q, 3Q+P_2\}$ which is Galois-stable.
Let $E_2=E_1/\langle P_2 \rangle$ and let $\phi: E_1 \rightarrow E_2$
be the induced isogeny. As $\Ker(\phi) \cap \langle Q \rangle=0$,
we see that $Q^\prime=Q+\langle P_2 \rangle$ has order $4$.
Moreover, the set $\{ Q^\prime, 3 Q^\prime\}$ is Galois-stable,
so $\langle Q^\prime \rangle$ is a $K$-rational cyclic subgroup
of order $4$ on $E_2$. The point of order $\ell$ on $E_1$
survives the isogeny, and so $E_2$ has a $K$-rational cyclic
subgroup
of order $4\ell$.}
 to an elliptic curve $E_2/K$ with
a $K$-rational cyclic subgroup isomorphic to $\Z/4\ell \Z$.
Thus we obtain a non-cuspidal $K$-point on the curves
$X_0(\ell)$, $X_1(\ell)$, $X_0(2\ell)$, $X_1(2\ell)$,
$X_0(4\ell)$, $X_1(2,2\ell)$. To achieve a contradiction
it is enough to show that there are no non-cuspidal $K$-points
on one of these curves. For small values of $\ell$, we have
found \texttt{Magma}'s \lq small modular curves package\rq,
as well as \texttt{Magma}'s functionality for computing Mordell--Weil
groups of elliptic curves over number fields invaluable.
Four of the modular
curves of interest to us happen to be elliptic curves.
The aforementioned \texttt{Magma} package gives
the following models:
\begin{gather}
\label{eqn:x020}
X_0(20) \; : \; 
y^2 = x^3 + x^2 + 4x + 4 \qquad \text{(Cremona label \texttt{20a1})},\\
X_0(14) \; : \;
y^2 + xy + y = x^3 + 4x - 6 \qquad \text{(Cremona label \texttt{14a1})},\\
\label{eqn:x011}
X_0(11) \; : \;
y^2 + y = x^3 - x^2 - 10x - 20 \qquad \text{(Cremona label \texttt{11a1})},\\
X_0(19) \; : \;
y^2 + y = x^3 + x^2 - 9x - 15 \qquad \text{(Cremona label \texttt{19a1})} .
\end{gather}

\begin{lem}\label{lem:35}
Let $p=5$. Then $\overline{\rho}_{E,\ell}$ is irreducible.
Moreover, $\overline{\rho}_{E^\prime,\ell}$ is irreducible.
\end{lem}
\begin{proof}
Suppose $\overline{\rho}_{E,\ell}$ is reducible
. By the above there is an elliptic 
curve $E_2$ over the quadratic field $K=\Q(\sqrt{5})$, 
with a $K$-rational subgroup isomorphic to 
$\Z/2\Z \times \Z/{2 \ell \Z}$.
From classification of torsion
subgroups of elliptic curves over quadratic fields 
\citep{Kamienny}
we deduce that $\ell \le  5$. However we are assuming
throughout that $\ell \ge 5$ and $\ell \ne p$.
This gives a contradiction as $p=5$.
Thus $\overline{\rho}_{E,\ell}$
is irreducible. The irreducibility of $\overline{\rho}_{E^\prime,\ell}$
follows from Lemma~\ref{lem:compare}.
\end{proof}

\begin{lem}\label{lem:p7}
Let $p=7$. Then $\overline{\rho}_{E,\ell}$ is irreducible.
\end{lem}
\begin{proof}
In this case $K$ is a cubic field.
By the 
classification of cyclic $\ell$-torsion on elliptic
curves over cubic fields \citep{Parent1,Parent2},
we know $\ell \le 13$. Since $\ell \ne p$,
we need only deal with the case $\ell=5$, $11$, $13$.
To eliminate $\ell=5$ and $\ell=11$ we computed the $K$-points 
on the
modular curves $X_0(20)$
and $X_0(11)$. 
These both have rank $0$ and their $K$-points
are in fact defined over $\Q$. 
The $\Q$-points of $X_0(20)$ are cuspidal thus $\ell \ne 5$.
The three non-cuspidal $\Q$-points on $X_0(11)$
all have integral $j$-invariants. As our curve $E$ 
has multiplicative reduction at $2\OO_K$, 
it follows that $\ell \ne 11$.

We suppose $\ell=13$. We now apply 
\citet[Theorem 1]{BN}. That theorem gives a useful and practical
criterion for ruling out the existence of 
torsion subgroups $\Z/m\Z \times \Z/n\Z$ on elliptic curves
over a given number field $K$ (the remarks at the end of 
Section 2 of \citep{BN} are useful when applying that theorem).
The theorem involves making certain choices and we indicate
them briefly; in the notation of the theorem, we take $A=\Z/26\Z$,
$L=\Q$, $m=1$, $n=26$, $X=X^\prime=X_1(26)$, $p=\fp_0=7$.
To apply the theorem we need the fact that the gonality
of $X_1(26)$ is $6$ \citep*{DH}, and that its Jacobian
has Mordell--Weil rank $0$ over $\Q$ \citep[page 11]{BN}.
We merely check that conditions (i)--(vi) of \cite[Theorem 1]{BN}
are satisfied, and conclude that there are no elliptic 
curves over $K$ with a subgroup isomorphic to $\Z/26\Z$.
This completes the proof.
\end{proof}

\begin{lem}
Let $p=11$. Then $\overline{\rho}_{E,\ell}$ is irreducible.
\end{lem}
\begin{proof}
Now $K$ has degree $5$. 
By the 
classification of cyclic $\ell$-torsion on elliptic
curves over quintic fields \citep{DKSS} we know
that $\ell \le 19$. As $\ell \ne p$ we need to
deal with $\ell=5$, $7$, $13$, $17$, $19$.

The elliptic
curves $X_0(20)$, $X_0(14)$  and
$X_0(19)$ have rank $0$ over $K$ and this allowed us to
quickly eliminate $\ell=5$, $7$, $19$. 

Suppose $\ell=13$. 
We again apply \citet[Theorem 1]{BN},
with choices $A=\Z/26\Z$,
$L=\Q$, $m=1$, $n=26$, $X=X^\prime=X_1(26)$, $p=\fp_0=11$
(with Mordell--Weil and gonality information as in
the proof of Lemma~\ref{lem:p7}).
This shows that there are no elliptic curves over $K$
with a subgroup isomorphic to $\Z/26\Z$,
allowing us to eliminate $\ell=13$.

Suppose $\ell=17$.
We apply the same theorem with choices $A=\Z/34\Z$,
$L=\Q$, $m=1$, $n=34$, $X=X^\prime=X_1(34)$, $p=\fp_0=11$.
For this we need the fact that $X$ has gonality $10$
\citep{DH} and that the rank of $J_1(34)$ over $\Q$
is $0$ \citep[page 11]{BN}. Applying the theorem shows
that there are no elliptic curves over $K$ with
a subgroup isomorphic to $\Z/34\Z$.
This completes the proof.
\end{proof}

It remains to deal with $p=13$. Unfortunately the field $K$
in this case is sextic, and the known bound \citep{DKSS}
for cyclic $\ell$-torsion over sextic fields is $\ell \le 37$,
and we have been unable to deal with the cases $\ell=37$ directly
over the sextic field.
We therefore proceed a little differently. We are in fact most
interested in showing the irreducibility of $\overline{\rho}_{E^\prime,\ell}$
where $E^\prime$ is the Frey curve 
defined over the degree $3$ subfield $K^\prime$.

\begin{lem}\label{lem:nmidx}
Let $p=13$.
Then $\overline{\rho}_{E^\prime,\ell}$ is irreducible.
\end{lem}
\begin{proof}
Suppose $\overline{\rho}_{E^\prime,\ell}$ is reducible.
We shall treat the case $13 \mid x$ and $13 \nmid x$ separately.
Suppose first that $13 \mid x$. Then the curve $E^\prime$ over
the field $K^\prime$ is semistable (Lemma~\ref{lem:Frey4}).
It is now straightforward to adapt the proof of Lemma~\ref{lem:tors}
to show that $E^\prime$ has non-trivial $\ell$-torsion,
or is $\ell$ isogenous to an elliptic curve with non-trivial
$\ell$-torsion. Thus there is an elliptic curve over $K^\prime$
with a point of exact order $2\ell$. Now $K^\prime$
is cubic, so by \citep{Parent1,Parent2} we have $\ell \le 13$.
As $\ell \ne p$, it remains to deal with the cases $\ell=5$, $7$, $11$.
The elliptic curves $X_0(14)$ and $X_0(11)$ have rank zero over $K^\prime$,
and in fact their $K^\prime$-points are the same as their $\Q$-points.
This easily allows us to eliminate $\ell=7$ and $\ell=11$ as before.
The curve $X_0(10)$ has genus $0$ so we need a different approach,
and we will leave this case to the end of the proof (recall that
$E^\prime$ does not necessarily have full $2$-torsion over $K^\prime$).

\medskip

Now suppose that $13 \nmid x$. Here $E^\prime/K^\prime$ is not
semistable. As we have assumed that $\overline{\rho}_{E^\prime,\ell}$ is 
reducible we have that $\overline{\rho}_{E,\ell}$ is reducible 
(Lemma~\ref{lem:compare}). Now we may apply Lemma~\ref{lem:tors}
to deduce the existence of $E_1/K$ (which is $E$ or $\ell$
isogenous to it) that has a subgroup isomorphic
to $\Z/2\Z \times \Z/2\ell\Z$. 
As before, let $\fp$ be the unique prime of $K$ above $13$.
By Lemma~\ref{lem:Frey1} the Frey curve $E$ has good reduction at $\fp$.
As $\fp \nmid 2 \ell$,
we know from injectivity of torsion that $4 \ell \mid \# E(\F_\fp)$.
But $\F_\fp=\F_{13}$.
By the Hasse--Weil bounds, $\ell \le (\sqrt{13}+1)^2/4\approx 5.3$. 
Thus $\ell=5$. 

It remains to deal with the case $\ell=5$ for both $13 \mid x$
and $13 \nmid x$. In both case we obtain a $K$-point on
$X=X_0(20)$ whose image in $X_0(10)$ is a $K^\prime$-point.
Would like to compute
$X(K)$. This computation proved beyond \texttt{Magma}'s
capability.
However, $K=K^\prime(\sqrt{13})$.
Thus the rank of $X(K)$ is the sum of the ranks of $X(K^\prime)$ and of 
$X^\prime(K^\prime)$ where $X^\prime$ is the quadratic twist
of $X$ by $13$. Computing the ranks of $X(K^\prime)$ and $X^\prime(K^\prime)$
turns out to be a task within the capabilities of \texttt{Magma},
and we find that they are respectively $0$ and $1$. Thus
$X(K)$ has rank $1$. With a little more work we find that
\[
X(K) = \frac{\Z}{6\Z} \cdot (4,10)  + \Z \cdot (3, 2 \sqrt{13}) \, .
\]
Thus $X(K)=X(\Q(\sqrt{13}))$. It follows that the  
$j$-invariant of $E^\prime$ must belong to $\Q(\sqrt{13})$.
But the $j$-invariant belongs to $K^\prime$ too, and
so belongs to $\Q(\sqrt{13}) \cap K^\prime=\Q$.

Let the rational integers 
$a$, $b$ be as in Sections~\ref{sec:Descent},~\ref{sec:FreyCurve}. Recall that
$b$ is odd, and that $\ord_2(a)=5n$ where $n>0$.
Write $a=2^{5n} a^\prime$ where $a^\prime$ is odd. 
We know that $\ord_2(j(E))=-(20n-4)$. The prime $2$
is inert in $K^\prime$. An explicit calculation,
making use of the fact that $a^\prime \equiv b \equiv 1 \pmod{2}$,
shows
that
\[
2^{20n-4} j(E) 
\equiv
\frac{\theta_j^2 \theta_k^2}{(\theta_j-\theta_k)^2} \pmod{2}. 
\]
Computing this residue for the possible values of $j$ and $k$,
we checked that it does not belong to $\F_2$, giving us a contradiction.
\end{proof}

\section{Proof of Theorem~\ref{thm:1}}\label{sec:final}

In Section~\ref{sec:peq3} we proved Theorem~\ref{thm:1}
for $p=3$. In this section we deal with the values
$p=5$, $7$, $11$, $13$. Let $\ell$, $m \ge 5$ be primes.
Suppose $(x,y,z)$ is a primitive non-trivial solution 
to \eqref{eqn:main}. Without loss of generality, $2 \mid x$.
We let $K=\Q(\zeta+\zeta^{-1})$ where $\zeta=\exp(2 \pi i/p)$.
For $p=13$ we also let $K^\prime$ be the unique subfield
of $K$ of degree $3$. Let $E$ be the Frey curve attached
to this solution $(x,y,z)$ defined
in Section~\ref{sec:FreyCurve} where we take $j=1$ and $k=2$.
For $p=13$ we also work with the Frey curve $E^\prime$
defined in Section~\ref{sec:FreyCurve2}
where we take $j=1$ and $k=5$ (these choices
satisfy the condition $\tau(\theta_j)=\theta_k$
where $\tau$ is unique involution on $K$).
By Lemma~\ref{lem:modularitysmallp} these
elliptic curves are modular. Moreover by the results
of Section~\ref{sec:Irred} the representation
$\overline{\rho}_{E,\ell}$ is irreducible for $p=5$, $7$, $11$, $13$,
and the representation $\overline{\rho}_{E^\prime,\ell}$ is irreducible
for $p=13$. 
Let $\cK$ be the number field $K$ unless $p=13$ and $13 \mid x$
in which case we take $\cK=K^\prime$. Also let $\cE$
be the Frey curve $E$ unless $p=13$ and $13 \mid x$
in which we take $\cE$ to be $E^\prime$. 
By Lemma~\ref{lem:ll} there is a Hilbert cuspidal
eigenform $\ff$ over $\cK$ of parallel weight $2$
and level $\cN$ as given in Table~\ref{table1}, such that
$\overline{\rho}_{\cE,\ell} \sim \overline{\rho}_{\ff,\lambda}$
where $\lambda \mid \ell$ is a prime of $\Q_\ff$, the field
generated by the Hecke eigenvalues of $\ff$.

Using the \texttt{Magma} \lq Hilbert modular
forms\rq\ package we computed the possible Hilbert newforms
at these levels. The information is summarized in Table~\ref{table1}.

\begin{table}[h]
\begin{center}
\begin{tabular}{|c|c|c|c|c|c|c|}
\hline
$p$ & Case &  Field $\cK$ &  Frey curve $\cE$ & Level $\cN$ & Eigenforms $\ff$ & $[\Q_\ff:\Q]$\\
\hline\hline
\multirow{2}{*}{$5$} & $5 \nmid x$ & $K$ & $E$ & $2\OO_K$ & -- & -- \\
\cline{2-7}
	& $ 5 \mid x$ & $K$ & $E$ & $2 \fp$ & -- & --\\
\hline\hline
\multirow{2}{*}{$7$} & $7 \nmid x$ & $K$ & $E$ & $2\OO_K$ & -- & -- \\
\cline{2-7}
	& $ 7 \mid x$ & $K$ & $E$ & $2 \fp$ & $\ff_1$ & $1$ \\
\hline\hline
\multirow{2}{*}{$11$} & $11 \nmid x$ & $K$ & $E$ & $2\OO_K$ & $\ff_2$ & $2$ \\
\cline{2-7}
	& $ 11 \mid x$ & $K$ & $E$ & $2 \fp$ & $\ff_3$, $\ff_4$ & $5$\\
\hline\hline
\multirow{5}{*}{$13$} & \multirow{3}{*}{$13 \nmid x$} 
& \multirow{3}{*}{$K$} & \multirow{3}{*}{$E$} & \multirow{3}{*}{$2\OO_K$} & $\ff_5$, $\ff_6$ & $1$ \\
& & &  & & $\ff_{7}$ & $2$ \\
& & &  & & $\ff_{8}$ & $3$ \\
\cline{2-7}
	&  \multirow{2}{*}{$13 \mid x$} & 
	\multirow{2}{*}{$K^\prime$} & \multirow{2}{*}{$E^\prime$} & \multirow{2}{*}{$2 \sB$} & $\ff_{9}$, $\ff_{10}$ & $1$\\
& & 	&  &  & $\ff_{11}$, $\ff_{12}$ & $3$\\
\hline
\end{tabular}
\caption{Frey curve and Hilbert eigenform information.
Here $\fp$ is the unique prime of $K$ above $p$,
and $\sB$ is the unique prime of $K^\prime$ above $p$.}
\label{table1}
\end{center}
\end{table}

As shown in the table, there are no newforms at the relevant
levels for $p=5$, completing the contradiction for this case.
\footnote{We point out in passing that for $p=5$ we could have also worked
with the Frey curve $E^\prime/\Q$.
In that case the Hilbert newforms $\ff$ are actually classical
newforms of weight $2$ and levels $2$ and $50$. There are
no such newforms of level $2$, but there are two newforms of level $50$
corresponding to the elliptic curve isogeny classes
\texttt{50a} and \texttt{50b}. These would require further work
to eliminate.}

We now explain how we complete the contradiction for the remaining
cases. 
Suppose 
$\fq$ a prime of $\cK$ such that $\fq \nmid 2 p \ell$. In particular, $\fq$ does not 
divide the level of $\ff$, and $\cE$ has good or multiplicative reduction at $\fq$.
Write $\sigma_\fq$ for a Frobenius element of $G_{\cK}$
at $\fq$. Comparing the traces of
$\overline{\rho}_{\cE,\ell}(\sigma_\fq)$ and $\overline{\rho}_{\ff,\lambda}(\sigma_\fq)$
we obtain:
\begin{enumerate}
\item[(i)] if $\cE$ has good reduction at $\fq$ then 
$a_\fq(\cE) \equiv a_\fq(\ff) \pmod{\lambda}$;
\item[(ii)] if $\cE$ has split multiplicative reduction at $\fq$ then
$\norm(\fq)+1 \equiv a_\fq(\ff) \pmod{\lambda}$;
\item[(iii)] if $\cE$ has non-split multiplicative reduction at $\fq$ then
$-(\norm(\fq)+1) \equiv a_\fq(\ff) \pmod{\lambda}$.
\end{enumerate}

\medskip

Let $q \nmid 2p \ell$ be a rational prime 
and let
\[
\cA_q=\{ (\eta,\mu   ) \; : \quad 0 \le \eta,~\mu \le q-1, \quad (\eta, \mu) \ne (0,0) \}.
\]
For $(\eta,\mu) \in \cA_q$ 
let
\begin{gather*}
u(\eta,\mu)=
\begin{cases}
(\theta_j+2)\eta^2+(\theta_j-2) \mu^2 & \text{if $p \nmid x$},\\
\frac{1}{(\theta_j-2)}\left((\theta_j+2)\eta^2+(\theta_j-2) \mu^2\right) & \text{if $p \mid x$},\\
\end{cases}\\
v(\eta,\mu)=
\begin{cases}
-\frac{(\theta_j-2)}{(\theta_k-2)}\left((\theta_k+2)\eta^2+(\theta_k-2) \mu^2\right) & \text{if $p \nmid x$},\\
-\frac{1}{(\theta_k-2)}\left((\theta_k+2)\eta^2+(\theta_k-2) \mu^2\right) & \text{if $p \mid x$}.\\
\end{cases}
\end{gather*}
Write 
\[
E_{(\eta,\mu)} \; : \; Y^2=X(X-u(\eta,\mu))(X+v(\eta,\mu)) \, .
\]
Let $\Delta(\eta,\mu)$, $c_4(\eta,\mu)$
and $c_6(\eta,\mu)$ be the usual invariants of this model. Let $\gamma(\eta,\mu)=-c_4(\eta,\mu)/c_6(\eta,\mu)$.
Let $(a,b)$ be as in Section~\ref{sec:Descent}. 
As $\gcd(a,b)=1$,
we have $(a,b) \equiv (\eta,\mu) \pmod{q}$ for some $(\eta,\mu)\in \cA_q$.
In particular $(a,b) \equiv (\eta,\mu) \pmod{\fq}$ for all primes $\fq \mid q$ of $\cK$.
From the definitions of the Frey curves $E$ and $E^\prime$ in Sections~\ref{sec:FreyCurve} and~\ref{sec:FreyCurve2}
we see that $\cE$ has good reduction at $\fq$ if and only if $\fq \nmid \Delta(\eta,\mu)$,
and in this case  $a_\fq(\cE)=a_\fq(E_{(\eta,\mu)})$.
Let
\[
B_\fq(\ff,\eta,\mu)=\begin{cases}
a_\fq(E_{(\eta,\mu)})-a_\fq(\ff) & \text{if $\fq \nmid \Delta(\eta,\mu)$},\\
\norm(\fq)+1-a_\fq(\ff) & \text{if $\fq \mid \Delta(\eta,\mu)$ and 
$\overline{\gamma(\eta,\mu)}
\in (\F_\fq^*)^2$},\\
\norm(\fq)+1+a_\fq(\ff) & \text{if $\fq \mid \Delta(\eta,\mu)$ and 
$\overline{\gamma(\eta,\mu)}
\notin (\F_\fq^*)^2$}.\\
\end{cases}
\] 
From (i)--(iii) above we see that
$\lambda \mid B_\fq(\ff,\eta,\mu)$. Now let
\[
B_q(\ff,\eta,\mu)=\sum_{\fq \mid q} B_\fq(\ff,\eta,\mu) \cdot \OO_{\ff},
\]
where $\OO_\ff$ is the ring of integers of $\Q_\ff$. 
Since $(a,b) \equiv (\eta,\mu) \pmod{\fq}$ for all $\fq \mid q$, we have that $\lambda \mid B_q(\ff,\eta,\mu)$.
Now $(\eta,\mu)$ is some unknown element of $\cA_q$. Let
\[
B_q^\prime(\ff)= \prod_{(\eta,\mu) \in \cA_q} B_q(\ff,\eta,\mu) \, .
\]
Then $\lambda \mid B_q^\prime(\ff)$. Previously, we have supposed that $q \nmid 2p \ell$.
This is inconvenient as $\ell$ is unknown. Now we simply suppose $q \nmid 2p$, and let $B_q(\ff)=q B_q^\prime(\ff)$.
Then, since $\lambda \mid \ell$, we certainly have  that $\lambda \mid B_q(\ff)$ regardless of whether $q=\ell$ or not.

\medskip

Finally, if $S=\{q_1,q_2,\dotsc,q_r\}$ is a set of rational primes, $q_i \nmid 2p$ then $\lambda$ divides
$\OO_\ff$-ideal $\sum_{i=1}^r B_{q_i}(\ff)$ and thus $\ell$ divides $B_S(\ff)=\norm(\sum_{i=1}^r B_{q_i}(\ff))$.
Table~\ref{table2} gives our choices for the set $S$ and the corresponding value of $B_S(\ff)$
for each of the eigenforms $\ff_1,\dotsc,\ff_{12}$ appearing in Table~\ref{table1}. Recalling
that $\ell \ge 5$ and $\ell \ne p$ gives a contradiction unless $p=13$ and $\ell=7$.
This completes the proof of Theorem~\ref{thm:1}.

The reader is wondering whether we can eliminate the case $p=13$ and $\ell=7$ by enlarging
our set $S$; here we need only concern ourselves with forms $\ff_9$ and $\ff_{11}$.
Consider $(\eta,\mu)=(0,1)$ which belongs to  $\cA_q$ for any $q$. The corresponding
Weierstrass model $E_{(0,1)}$ is singular with a split note. It follows that 
$B_{\fq}(\ff,0,1)=\norm(\fq)+1-a_\fq(\ff)$. Note that if $\lambda$ is a prime of $\Q_\ff$
that divides $\norm(\fq)+1-a_\fq(\ff)$ for all $\fq \nmid 26$, then $\ell$ will
divide $B_S(\ff)$ for any set $S$ where $\lambda \mid \ell$. This appears to be the 
case  with $\ell=7$ for $\ff_{11}$, and we now show that it is indeed  the case for $\ff_9$.
Let $F$ be the elliptic curve with Cremona label \texttt{26b1}:
\[
F \; : \; y^2 + xy + y = x^3 - x^2 - 3x + 3,
\]
which has conductor $2 \sB$ as an elliptic curve over $\cK$.
As $\cK/\Q$ is cyclic, we know that $F$ is modular over $\cK$ and hence corresponds
to a Hilbert modular form of parallel weight $2$ and level $2 \sB$, and by comparing
eigenvalues we can show that it in fact corresponds to $\ff_9$. Now the point $(1,0)$
on $F$ has order $7$. It follows that $7 \mid \# E(\F_\fq)=\norm(\fq)+1-a_\fq(\ff)$
for all $\fq \nmid 26$ showing that for $\ff_9$ we can never eliminate $\ell=7$
by enlarging the set $S$. We can still complete the contradiction in this
case as follows. Note that $\overline{\rho}_{\ff_9,7}\sim \overline{\rho}_{F,7}$ which
is reducible. As $\overline{\rho}_{\cE,7}$ is irreducible we have $\overline{\rho}_{\cE,7} \not \sim
\overline{\rho}_{\ff_9,7}$, completing the contradiction for $\ff=\ff_9$. 
We strongly suspect that reducibility of $\overline{\rho}_{\ff_{11},\lambda}$
(where $\lambda$ is the unique prime above $7$ of $\Q_{\ff_{11}}$)
but we are unable to prove it.

\begin{table}[h]
\begin{center}
{\tabulinesep=0.8mm
\begin{tabu}{|c|c|c|c|c|}
\hline
$p$ & Case &  $S$ & Eigenform $\ff$ & $B_S(\ff)$\\
\hline\hline
$7$ & $ 7 \mid x$ & $\{3\}$ & $\ff_1$ & $2^8\times 3^5 \times 7^6$ \\
\hline\hline
\multirow{3}{*}{$11$} & $11 \nmid x$ & $\{23,43\}$ & $\ff_2$ & $1$ \\
\cline{2-5}
& \multirow{2}{*}{$11 \mid x$} & \multirow{2}{*}{$\{23,43\}$} & $\ff_3$ & $1$ \\
\cline{4-5}
&  & & $\ff_4$ & $1$ \\
\hline\hline
\multirow{8}{*}{$13$} & \multirow{4}{*}{$13 \nmid x$} & \multirow{4}{*}{$\{79, 103\}$} & $\ff_5$ & 
$2^{6240}\times 3^{312}$ \\
\cline{4-5}
& & & $\ff_6$ & $ 2^{12792}\times 3^{234}$\\
\cline{4-5}
& & & $\ff_7$ & $2^{10608}\times 3^{624}$\\
\cline{4-5}
& & & $\ff_8$ & $2^{18720}\times 3^{936}$\\
\cline{2-5}
& \multirow{4}{*}{$13 \mid x$} & \multirow{4}{*}{$\{3, 5, 31, 47\}$} & $\ff_9$ & $7^2$ \\
\cline{4-5}
& & & $\ff_{10}$ & $3^7$\\
\cline{4-5}
& & & $\ff_{11}$ & $7^6$\\
\cline{4-5}
& & & $\ff_{12}$ & $1$\\
\hline
\end{tabu}
}
\caption{Our choice of set of primes $S$ and the value of $B_S(\ff)$ for 
each of the eigenforms in Table~\ref{table1}.
}
\label{table2}
\end{center}
\end{table}

\medskip

\noindent \textbf{Remark.}
We now explain why we believe that the above strategy will succeed in proving
that that \eqref{eqn:main} has no non-trivial primitive solutions,
or at least in bounding the exponent $\ell$, for larger values of $p$
provided the eigenforms $\ff$ at the relevant levels can be computed. 
The usual obstruction, c.f.\ \citep[Section 9]{IHP}, to bounding the exponent comes from eigenforms $\ff$
that correspond
to elliptic curves with a torsion structure that matches the Frey curve $\cE$. Let $\ff$ be such an eigenform.
Let $q \nmid 2p$ be a rational prime and $\fq_1,\dotsc,\fq_r$ be the primes of $\cK$ above $q$. 
Note that $\norm(\fq_1)=\dotsc=\norm(\fq_r)=q^{d/r}$ where $d=[\cK: \Q]$. 
We would like to estimate the \lq probability\rq\ that $B_q(\ff)$ is non-zero.
Observe that if $B_q(\ff)$ is non-zero, then we obtain a bound for $\ell$. Examining the definitions
above, shows that the ideal $B_q(\ff)$ is $0$ if and only if there is some $(\eta,\mu) \in \cA_q$
such that $a_\fq(E_{(\eta,\mu)})=a_\fq(\ff)$ for $\fq=\fq_1,\fq_2,\dotsc,\fq_r$.
Treating $a_\fq(E_{(\eta,\mu)})$ as a random variable belonging to the Hasse interval $[-2 q^{d/2r}, 2 q^{d/2r}]$,
we see that the \lq probability\rq\ that $a_\fq(E_{(\eta,\mu)})=a_\fq(\ff)$ is 
roughly $c/q^{d/2r}$ with $c=1/4$. We can be a little more sophisticated and take
account of the fact that the torsion structures coincide, and that these impose congruence restrictions
on both traces. In that case we should take $c=1$ if $\cE$ has full $2$-torsion (i.e.\ $\cE$ is
the Frey curve $E$) and take $c=1/2$ if $\cE$ has
just one non-trivial point of order $2$ (i.e.\ $\cE=E^\prime$
and $p \equiv 1 \pmod{4}$). Thus the \lq probability\rq\ that $a_\fq(E_{(\eta,\mu)})=a_\fq(\ff)$ 
for all $\fq \mid q$ simultaneously is roughly $c^r/q^{d/2}$.
Since $B_q(\ff)=q \prod_{(\eta,\mu)\in \cA_q} B_q(\ff,\eta,\mu)$.
It follows that the \lq probability\rq\ $\mathbb{P}_q$ (say) that $B_q(\ff)$
is non-zero satisfies
\[
\mathbb{P}_q \sim \left(1 - \frac{c^r}{q^{d/2}}\right)^{q^2-1}
\]
For $q$ large, we have $(1-c^r/q^{d/2})^{q^{d/2}} \approx e^{-c^r}$. 
For $d\ge 5$, from the above estimates, we expect that
$\mathbb{P}_q \rightarrow 1$ as $q \rightarrow \infty$. Thus
we certainly expect our strategy to succeed in bounding the exponent $\ell$.

\bibliographystyle{plainnat}
\bibliography{genferm}

\end{document}